\documentclass[reqno]{amsart}
\usepackage[utf8]{inputenc}

\usepackage{geometry}
 \usepackage{mathrsfs}
\usepackage{amsmath,amsthm,amssymb,mathabx,amsfonts}
\usepackage{bbm}
\usepackage{graphicx}
\usepackage{subcaption}
\usepackage{todonotes}
\usepackage{cleveref}
\usepackage{slashed}
\usepackage{color}
\usepackage[style=ext-numeric,sorting=nyt, url = false, doi = false, eprint = false, isbn = false,giveninits=true, articlein=false,uniquename=init]{biblatex}

\usepackage{cancel}
\numberwithin{equation}{section}

\newtheorem{theorem}{Theorem}[section]

\newtheorem{lemma}[theorem]{Lemma}
\newtheorem{cor}[theorem]{Corollary}
\newtheorem{remark}[theorem]{Remark}
\newtheorem{conjecture}[theorem]{Conjecture}

\newcommand{\bp}{{\bf p}}

\newcommand{\bv}{{\bf v}}

\newcommand{\eps}{\varepsilon}

\newcommand{\re}{\mathrm{Re} \, }

\newcommand{\R}{\mathbb{R}}
\newcommand{\N}{\mathbb{N}}
\newcommand{\Z}{\mathbb{Z}}
\newcommand{\C}{\mathbb{C}}

\DeclareMathOperator{\supp}{supp}
\DeclareMathOperator{\sinc}{sinc}
\DeclareMathOperator{\Ai}{{\rm Ai}}

 \newcommand{\rev}[1] {#1}

\title[Estimates for Forced Dirac]{On the Time-decay of solutions arising from periodically forced Dirac Hamiltonians}

\author{Joseph Kraisler}
\address{Department of Mathematics and Statistics, Amherst College, Amherst, MA 01002, USA}
\email{jkraisler@amherst.edu}
\author{Amir Sagiv}
\address{Department of Mathematical Sciences, New Jersey Institute of Technology, University Heights, Newark, NJ 07102, USA}
\email{amir.sagiv@njit.edu}
\author{Michael I.\ Weinstein}
\address{Department of Applied Physics and Applied Mathematics and Department of Mathematics, Columbia University, 500 W120 Street, New York, NY 10027, USA}
\email{miw2103@columbia.edu}

\begin{document}

\begin{abstract} 

There is increased interest in time-dependent (non-autonomous) Hamiltonians, stemming in part from the active field of Floquet quantum materials. Despite this, dispersive time-decay bounds, which reflect energy transport in such systems, have received little attention.
 
We study the dynamics of non-autonomous, time-periodically forced, Dirac Hamiltonians: $i\partial_t\alpha =\slashed{D}(t)\alpha$, where $\slashed{D}(t)=i\sigma_3\partial_x+ \nu(t)$ is time-periodic but not spatially localized. For the special case  $\nu(t)=m\sigma_1$, which models a  relativistic particle of constant mass $m$, 
one has a dispersive decay bound: $\|\alpha(t,x)\|_{L^\infty_x}\lesssim t^{-\frac12}$.
Previous analyses  of Schr\"odinger Hamiltonians (e.g. \cite{beceanu2011new, galtbayar2004local, goldberg2009strichartz})   suggest that this decay bound persists for small, spatially-localized and time-periodic $\nu(t)$. However, we show that this is not necessarily the case if  $\nu(t)$  is not spatially localized. Specifically, we study two non-autonomous Dirac models whose time-evolution (and monodromy operator) is constructed via Fourier analysis. In a rotating mass model, the dispersive decay bound is of the same type as for the constant mass model. However, in a model with a periodically alternating sign of the mass, the results are quite different.
 By stationary-phase analysis of the associated Fourier representation, we display initial data for which the $L^\infty_x$ time-decay rate are considerably slower:  $\mathcal{O}(t^{-1/3})$ or even $\mathcal{O}(t^{-1/5})$ as $t\to\infty$. 
\end{abstract}

\maketitle

\section{Introduction}

We study the dynamics of non-autonomous, time-periodically forced Dirac equations
\begin{subequations}
\label{eq:tperDirac}
\begin{align}
              i\partial_t \alpha (t,x) &= \left( i\sigma _3 \partial_x + \nu (t) \right) \alpha(t,x) \, ,  \label{eq:Dirac}\\
               \alpha(0,x) &= f \in L^2 (\R ;\C ^2 ) \, \label{eq:data} ,
\end{align}
\end{subequations}
Here $\nu(t)$ is a bounded  $T-$ periodic $2\times 2$ Hermitian 
 matrix-valued function, and $\sigma_3$ is the standard  Pauli matrix; see \eqref{eq:pauli}. Note that the $L^2(\mathbb R)$ norm is constant along solutions of 
 \eqref{eq:tperDirac}, i.e., \begin{equation}
\|\alpha(t,\cdot)\|_{L^2} = \|f\|_{L^2} \, ,
 \label{eq:L2}\end{equation}
for all $t\geq 0$. We investigate (for different choices of $\nu(t)$)  whether, in what sense, and at what rate, solutions of the initial value problem
\eqref{eq:tperDirac} decay as time advances.

The simplest cases are perhaps misleading: when $\nu(t)$ commutes with $\sigma_3\partial_x$, i.e., when $\nu(t)$ is diagonal, then
\begin{equation}
\alpha(t,\cdot) = e^{i\int_0^t\nu(s)ds}\ e^{\sigma_3 \partial_x t} \ f \, ,
\label{eq:factored}\end{equation}
and the components of $\alpha$ are right- and left- traveling waves, each multiplied  by a time-dependent phase. Each traveling wave propagates to infinity without distortion.
Clearly \eqref{eq:L2} still holds, but the solution does not exhibit {\it dispersive time-decay}, e.g. a decay of its $L^\infty(\mathbb R)$ norm as $t\to\infty$.\footnote{It is true, however, that for any spatially-localized initial data $f$, the amplitude tends to zero on any fixed compact set, due to the non-autonomous version of RAGE theorem \cite{enss1983bound}. This result, however, does not describe the rate of decay.} 

Our goal in this paper is to  study time-parametrically forced Dirac equations \eqref{eq:tperDirac} which exhibit dispersive decay and, in particular, to develop a quantitative understanding of the possible rates of decay. In view of the above example, {\it we focus on  cases where} 
\begin{equation} [\nu(t), i\partial_x \sigma _3] \neq 0 \, , \label{eq:nuOFt}\end{equation} 
for a non-negligible set of $t\in [0,T]$.
 Best known is the case 
$\nu(t)=m\sigma_1$, which models a relativistic particle of constant mass $m$. Noncommutativity of $\sigma_1$ and $\sigma_3$ forbids factorization as in \eqref{eq:factored}, but the initial-value problem can nevertheless be solved in this constant coefficient case by Fourier transform; if the initial conditions are sufficiently smooth, then  one has ``dispersive time-decay estimate''   $\|\alpha(x,t)\|_{L^\infty_x}\lesssim t^{-\frac12}$ \cite{Erdogan21}. 

{\it But what if $\nu(t)$ is \underline{non-constant} in time and does not commute with $\sigma_3$?} This class of models arises
in as the effective (homogenized) dynamics  of Floquet materials, an emergent and very active area in the fields of condensed matter physics \cite{cayssol2013floquet}, photonics \cite{ozawa2019topological}, and acoustics \cite{xue2022topological}; see the discussion below in Section \ref{sec:phys_mot}. An understanding of dispersive decay rates for this class of Hamiltonians, $\slashed{D}(t)$, appears to be open. 

The question of dispersive decay bounds in  {\em autonomous} Hamiltonian dynamics has been studied extensively, e.g. for time-independent Schr{\"o}dinger Hamiltonians \cite{jensen1979spectral, komech2010weighted, kopylova2014dispersion,schlag2005dispersive} and Dirac Hamiltonians \cite{d2005decay, burak2019limiting, erdougan2021massless, erdougan2019dispersive, erdougan2018dispersive, ERDOGAN2DMassive, green2024massless, kovavrik2022spectral, kraisler2023dispersive}. 
 Much less  is known for  {\em non-autonomous} Hamiltonians. All existing results, to the best of our knowledge, concern Schr{\"o}dinger equations in dimensions $d\geq 3$, and crucially all in the regime where the time-dependent term is a {\em perturbation of an autonomous Hamiltonian}, in some sense \cite{beceanu2011new,beceanu2012schrodinger, galtbayar2004local, goldberg2009strichartz, rodnianski2004time}. In these settings, the authors recover ``autonomous-like'' decay rates in non-autonomous settings.   See Sec.\ \ref{sec:lit} for a more detailed review. Such methods are not expected to work for those cases where $\nu(t)$ is not localized in space.

 We next introduce two solvable  models where $\slashed{D}(t+T)=\slashed{D}(t)$ and $\nu(t)$ is not spatially localized, for which we can obtain dispersive estimates. For one of which, the decay rates are substantially slower compared to its autonomous analogs.

\subsection{Models}\label{sec:models}  

\subsection*{Sign-switching mass}

Consider  $\nu(t)$ which ``switches'' discontinuously and periodically between positive and negative masses, i.e., $$\nu(t) = \begin{cases}
     m\sigma_1 \, ,& t\in \left[jT, (j+\frac12 )T \right)  , \\
     -m\sigma_1 \, , & t\in \left[(j+\frac12)T, (j+1)T\right) ,
\end{cases} \,, \forall j\in \Z \, , \qquad m,T>0 \, .$$
Theorem \ref{thm:maintheorem} shows that, in sharp contrast with the theory of autonomous Dirac operators, the time-decay is {\em at most} of rate $t^{-1/3}$. Furthermore, for special choices of the mass parameter $m>0$, the rate is exceptionally slow, at most $t^{-1/5}$ (Theorem~\ref{thm:15theorem}).
%

\subsection*{Complex rotation}

The second model we consider is a time-periodic ``rotating mass'':
\begin{equation}\nu(t) = m\begin{pmatrix}
    0 & e^{i\omega t}\\ e^{-i\omega t} & 0 
\end{pmatrix} = m\left[\cos(\omega t)\sigma_1 - \sin(\omega t)\sigma_2\right] ,\label{eq:rotation}
\end{equation}
where $m$ and $\omega$ are  real positive constants. 
The dynamics associated with \eqref{eq:rotation} can be mapped to the (non-rotating) massive Dirac equation (Theorem \ref{Thm:timeharmonic-prop}), which yields a $t^{-1/2}$ decay rate (Corollary \ref{cor:thdecay}), as in the constant mass Dirac equation \cite{Erdogan21}. Thus, here is an example of a time-dependent (and non-localized) $\nu(t)$, which satisfy the non-commutation condition \eqref{eq:nuOFt}, but exhibits an autonomous-like decay rate nonetheless.

Finally, we note that the case \eqref{eq:rotation} is a time-periodic variant  of an effective Hamiltonian, which was derived and studied in the study of defect modes in dislocated media  \cite{fefferman2014topologically,fefferman2017topologically,drouot2020defect,drouot2021bec}.

\subsection{Physical motivation}\label{sec:phys_mot}

Dirac equations were first introduced to provide a relativistic framework for quantum mechanics \cite{erdogan1782dirac, thaller2013dirac}. Relevant to this work is the fact that (autonomous) Dirac Hamiltonians also arise in study of periodic (crystalline) media. Specifically, as the {\it effective (homogenized) Hamiltonians} describing wave-packets in periodic structures that are spectrally concentrated near {\em Dirac points} -- linear (in 1D) or conical (in 2D)  degeneracies in the band structure, a phenomenon occuring in graphene and related  quantum/condensed-matter settings \cite{ castro2009electronic, ablowitz2009conical,  drouot2020defect, fefferman2014wave, fefferman2017topologically, fefferman2014topologically}.

Recently, there has been a significant experimental and theoretical progress in the study of {\it Floquet materials},  crystalline materials whose effective transport properties are controlled by time-periodic driving. Within the class of Floquet materials are 
 {\em Floquet topological insulators}, which exhibit changes in topological phase in response to appropriate time-forcing \cite{cayssol2013floquet, rudner2020band}. In such materials, non-autonomous Dirac equations are the appropriate low-energy/homogenized model \cite{ablowitz2017tight, bal2022multiscale, hameedi2023radiative, sagiv2022effective, sagiv2023near}. 
 In a class of physically relevant models,   the time-periodic driving is {\em uniform in space.}\footnote{ 
  An example of an experimental setting is the study of  electronic conductance  is in materials such as the hexagonal quantum material graphene. The spatial support of the time-periodic forcing corresponds to the finite region of the material, where an external laser beam drives its electrons \cite{mciver2020light, perez2014floquet, wang2013observation}. Our Hamiltonian reflects the modeling assumption that  uniform time-dependent forcing is applied to an area which is large compared with the material's lattice constant.}
   Thus, the natural models in the context of Floquet media are those where the time-periodic forcing  cannot be thought of as localized in space, as in \cite{beceanu2011new, beceanu2012schrodinger, galtbayar2004local, goldberg2009strichartz}. 
Finally, we remark that dispersive time-decay rates play a role in analysis of the metastability of bound states when subjected to parametric (periodic or more general) forcing. See, for example, \cite{soffer1998nonautonomous,hameedi2023radiative} for perturbative analyses of general models and \cite{borrelli2022complete, costin2008ionization,  costin2010ionization, costin2018nonperturbative, costin2001evolution} for non-perturbative studies in exactly solvable Schr{\"o}dinger time-periodic Hamiltonians.

\subsection{Broader discussion of literature on dispersive time-decay bounds }\label{sec:lit}

Dispersive decay estimates is a standard topic in the analysis of PDEs. The literature concerning {\em autonomous} (time-independent) Schr{\"odinger} or Dirac equations is extensive; see \cite{jensen1979spectral, komech2010weighted, kopylova2014dispersion,schlag2005dispersive} and the references therein. To the best of our knowledge, there has been no work on dispersive decay estimates in time-dependent Dirac equations, in any spatial dimension.

Time-decay estimates for  autonomous {\em Dirac} equations received extensive attention; see, for example, \cite{d2005decay, burak2019limiting, erdougan2021massless, erdougan2019dispersive, erdougan2018dispersive, ERDOGAN2DMassive, green2024massless, kovavrik2022spectral, kraisler2023dispersive}. 
Such estimates play a role in the weakly nonlinear scattering and stability theory of  {\em semilinear} Dirac equations \cite{boussaid2019nonlinear, pelinovsky2012asymptotic}.
 Most relevant to this work is the massive one-dimensional Dirac equation studied by Erdogan and Green \cite{Erdogan21}: denoting by $P$ the $L^2 (\R ;\C^2)$ projection onto the continuous spectral part of a massive Dirac operator with a rapidly decaying potential $D\equiv i\sigma_3\partial_x + m\sigma_1 +V(x)$, then
$$ \left\|e^{iDt}P \langle D\rangle^{-\frac32-\varepsilon} \right\|_{L^1 \to L^{\infty}} \lesssim t^{-\frac12} \, ,$$
for every $\varepsilon > 0$, where $\langle D\rangle^{-\frac12-\varepsilon}$ is a smoothing operator, defined via functional calculus. Furthermore, an improved $t^{-\frac32}$ holds in the generic case where no threshold resonances exist \cite{Erdogan21}.

There are significantly fewer results on time-decay bounds for {\em non-autonomous} dispersive equations. All concern  Schr{\"o}dinger Hamiltonians $H(t)=-\Delta + V(t,x)$, in dimensions $d\geq 3$, where the $H(t)$ is a small and spatially localized perturbation of a Schr\"odinger operator $H^0=\Delta+U(x)$ \cite{beceanu2011new, beceanu2012schrodinger, galtbayar2004local, goldberg2009strichartz, rodnianski2004time}. In settings where time-decay bounds for $\exp(-iH^0t)$ restricted to its continuous spectral part are known, persistence of these time-decay bounds is proved by perturbation theory. The analysis is based on a study of the {\em Floquet Hamiltonian} \cite{howland1979scattering}: $K\equiv i\partial_t - H(t)$, acting on functions of both space and time.
In our specific setting, we explicitly construct the monodromy operator  as a  Fourier integral (see, e.g. \eqref{eq:Mnf_integral}) and study it by oscillatory integral methods.
 Further, the perturbative approach  does not apply in our setting since our time-dependent perturbation is not spatially localized.

\subsection{Structure of the paper}

The main results (Theorems \ref{thm:maintheorem}, \ref{thm:15theorem}, and \ref{Thm:timeharmonic-prop}) are presented in Section \ref{sec:mainRes}, together with numerical observations and conjectures. We present key notations in Sec.\ \ref{sec:notation}. The proofs of Theorem \ref{thm:maintheorem} and Theorem \ref{thm:15theorem} are presented in Sec.\ \ref{sec:mainPf}, followed by the proof of Theorem \ref{Thm:timeharmonic-prop} in Sec.\ \ref{sec:thpf}.

\subsection*{Acknowledgements}
This research was supported in part by National Science Foundation grants 
DMS-1908657 and DMS-1937254 (MIW),  Simons Foundation Math + X Investigator Award \#376319 (MIW), and the Binational Science Foundation Research Grant \#2022254 (MIW, AS). AS would like to thank the Department of Applied Physics and Applied Mathematics at Columbia University for hosting him during the writing of this manuscript.

\section{Models, approaches, and main results}\label{sec:mainRes}

\subsection{The mass-switching model}

First, consider a switching-mass model for $\alpha(t,x)\in L^2(\mathbb R;\mathbb C)$ of the form
\begin{align}\label{eq:pwc_mdirac}
     i\partial_t \alpha = \left( i\sigma _3 \partial_x + \sigma_1 \nu (t) \right) \alpha \, , \qquad \nu(t) = \begin{cases}
     m \, ,& t\in \left[2j, 2j+1  \right)    \, , \\
     -m \, , & t\in \left[2j+1, 2j+2\right) \, ,
\end{cases} \,  j\in \Z \, ,
\end{align}
where $m>0$ denotes a  ``mass'' parameter. 
We denote by $\mathcal{U}(t)$  the solution operator for the dynamical system \eqref{eq:pwc_mdirac}. 

The Hamiltonian $H(t)=i\sigma _3 \partial_x + \sigma_1 \nu (t)$ is  periodic in $t$, and without loss of generality we take the period time to be $T=2$.

Hence, the dynamics are determined by dynamics by the {\it monodromy operator}, $M=\mathcal{U}(2)$,
 which maps data at $t=0$, $\alpha(0)=f\in L^2(\mathbb R;\mathbb C)$, to the solution $\alpha(2)=Mf\in L^2(\mathbb R;\mathbb C)$, at time $t=2$.
 
 Let $\mathcal{U}_+(t)$ denote the solution map for the IVP on the interval $0\le t<1$ and let 
  $\mathcal{U}_-(t)$ the solution map
   for the IVP on the interval $1\le t<2$.
    Then, \begin{equation}
    Mf \equiv \mathcal{U}(2)= \mathcal{U}_-(1)\mathcal{U}_+(1).
    \label{eq:Mdef-switch}
    \end{equation}
    Further, 
    $\mathcal{U}(2n) = M^n$ for all $ n\in \mathbb{Z}$
Note that $M$ is a unitary operator on $L^2(\R;\C^2)$. 

In Section \ref{app:monodromy-derivation} we use that $H(t)$ is invariant continuous translations in $x$ to  express $M$ via the Fourier transform:
\begin{align} \label{eq:monodromy}
    (Mf)(x) = \frac{1}{2\pi}\int_{\R}  P(\xi;m) \begin{pmatrix}
     e^{+2i\theta(\xi;m)}& 0 \\
     0&  e^{-2i \theta(\xi;m)}\end{pmatrix} P^* (\xi;m) \hat{f}(\xi)e^{i\xi x} \, d\xi \, .
\end{align}
Here, $\hat{f}$, the Fourier transform of $f\in L^2(\R;\C^2)$ is given by \eqref{eq:FourierDef}.
The expression \eqref{eq:monodromy} involves a phase function or ``dispersion relation'' $\theta(\xi)$ is given by
\begin{align}\label{eq:theta}
     \theta(\xi; m)=\arctan \left( \frac{\xi \sin (\omega (\xi))}{\sqrt{m^2 + \xi ^2 \cos^2(\omega(\xi)) }} \right)\, ,\qquad \omega(\xi;m) = \sqrt{m^2+\xi^2}\, 
\end{align}
and $P(\xi;m)$ a $2\times 2$ unitary matrix with entries displayed  in~\eqref{eq:Pmatrix-elements}. Thus, bounding $\|
\mathcal{U}(2n)\|_{L^1\to L^{\infty}}$ for $n\gg 1$ boils down to analyzing the rapidly oscillatory integral 
\begin{align}\label{eq:Mnf_integral}
(M^n f)(x) = \frac{1}{2\pi}\int_{\R} P(\xi;m) \begin{pmatrix}
     e^{+2 i n\theta(\xi;m) }& 0 \\
     0& e^{-2 i n\theta(\xi;m)}\end{pmatrix} P^* (\xi;m) \hat{f}(\xi)e^{i\xi x} \, d\xi \, .
\end{align}

\noindent

We study \eqref{eq:Mnf_integral} by  stationary phase methods for initial data $f$, where $\hat{f}$ is supported in a neighborhood of  momentum $\xi=0$. A key tool is 
van der Corput's Lemma (\Cref{lem:vandercorput}), which relates time-decay bounds  to lower bounds on derivatives of the phase function, $\theta(\xi,m)$. We shall see that the dependence of $\theta(\xi,m)$ on the mass  parameter $m$ is such that decay bounds for $M^nf$ depend on $m$. 
Let $\Sigma\subset(0,\infty)$ be the vanishing set of the third derivative of $\theta(0;m)$
\begin{align}\label{eq:SigmaDef}
    \Sigma \equiv \{ m\in (0,\infty)\ | \ \theta'''(0;m)=0\}\, .
\end{align}

Lemma \ref{lem:thetathirdder0} below states that the set  $\Sigma$ is discrete, and hence  the condition $m\notin \Sigma$ is  generic.
Our main result is that for $m\notin \Sigma$, the sharp dispersive decay rate is {\em at most} $t^{-1/3}$. Hence,  for this piecewise constant mass-switching model, sharp time-decay rates must be slower than those associated with the constant mass Dirac equation.

\begin{theorem}\label{thm:maintheorem}
    Consider the system \eqref{eq:pwc_mdirac} with the monodromy operator $M$, as in \eqref{eq:monodromy}, with $m\notin \Sigma$, see \eqref{eq:SigmaDef}. Let $f\in\mathcal{S}(\R;\C^2)$ be any function with Fourier transform supported in a sufficiently small ($m-$ dependent)  neighborhood of the origin. Then for all $n\geq 1$
    \begin{align}\label{eq:main_ubd}
        \| M^n f\|_{L^{\infty}} \lesssim \frac{1}{n^{1/3}}\|f\|_{L^{1}}\ .
    \end{align}
     Furthermore, for such initial data $f$, there exist $s_0, C \neq 0$ such that for $x_n \equiv ns_0$,
    \begin{align}\label{eq:main_exp}
         \vert(M^{n} f)(x_n)\vert =  C\begin{pmatrix}
    \int\limits_{\R} f_1(x) \, dx \\ \int\limits_{\R} f_2 (x) \, dx
\end{pmatrix} \frac{1}{n^{1/3}} + \mathcal{O}\left(\frac{1}{n^{2/3}}\right)  \, ,
    \end{align}
    where the leading error term  depends on $\partial_{\xi}\hat{f}(\xi=0)$. 
\end{theorem}
The following result shows that for for mass parameter values in the discrete set $\Sigma$, an even slower rate of decay is attained  for the same collection of initial data.
\begin{theorem}\label{thm:15theorem}
    Assume the setup of Theorem \ref{thm:maintheorem}, but now with $m\in\Sigma$. Then for all $n\geq 1$
    \begin{align}\label{eq:15bound}
        \| M^n f\|_{L^{\infty}} \lesssim \frac{1}{n^{1/5}}\|f\|_{L^{1}}\ .
    \end{align}
     Furthermore, for such initial data $f$, there exist constants $C, s_1 \neq 0$ such that for $x_n \equiv ns_1 $ 
    \begin{align}\label{eq:15expand}
        \left| (M^{n} f)(x_n)\right| = C\begin{pmatrix}
    \int\limits_{\R} f_1(x) \, dx \\ \int\limits_{\R} f_2 (x) \, dx
\end{pmatrix}\frac{1}{n^{1/5}} + \mathcal{O}\left(\frac{1}{n^{2/5}}\right) \, .
    \end{align}
\end{theorem}

\begin{remark}\label{rem:smoothed-bound}
 Since the initial conditions  $f$  appearing in Theorems \ref{thm:maintheorem} and \ref{thm:15theorem} are in Schwartz class, the time-decay bounds \eqref{eq:main_ubd} and \eqref{eq:15bound} hold with $M^n$ replaced by $M^n\langle\partial_x\rangle^{-r}$ for any $r\ge0$. Thus, any general dispersive time-decay bound must have a rate slower or equal to $t^{-1/3}$ and $t^{-1/5}$, depending on the mass parameter.
\end{remark}

\subsection{Numerical observations and conjectures}\label{sec:num}

Consider an oscillatory integral such as \eqref{eq:Mnf_integral} which depends on a parameter $n$.  The behavior of the phase function impacts the  decay of the integral as $n$ tends to infinity. If the phase function is linear, then the Riemann-Lebesgue Lemma ensures that the integral tends to zero as $n$ tends to infinity, with however no information on the rate of decay. On the other hand,  if the phase has critical points, or at distinguished points   derivatives of the phase of higher order vanish, then one can apply the method of stationary phase, or more generally Van der Corput's Lemma \ref{lem:vandercorput} and obtain a  quantitative information on the decay rate. 
The $t^{-1/3}$ or $t^{-1/5}$ time-decay rates for a generic and exceptional masses, respectively, are due to the existence of an inflection point in the dispersion relation $\theta(\xi)$, see \eqref{eq:theta}. It is straightforward to show that $\theta '' (0;m)=\theta ^{(4)}(0;m)=0$ for all $m$ values (Appendix \ref{app:thetaderivatives}). Additionally, Fig.\ \ref{fig:theta3at0} shows that $\theta '''(0;m)=0$ for a discrete set of ``exceptional'' $m$ values, an assertion that can be verified by using the continuity of $\theta '''(0,m)$.
 \begin{figure}[h]
     \centering
     \begin{subfigure}{0.45\textwidth}
        \includegraphics[width=\linewidth]{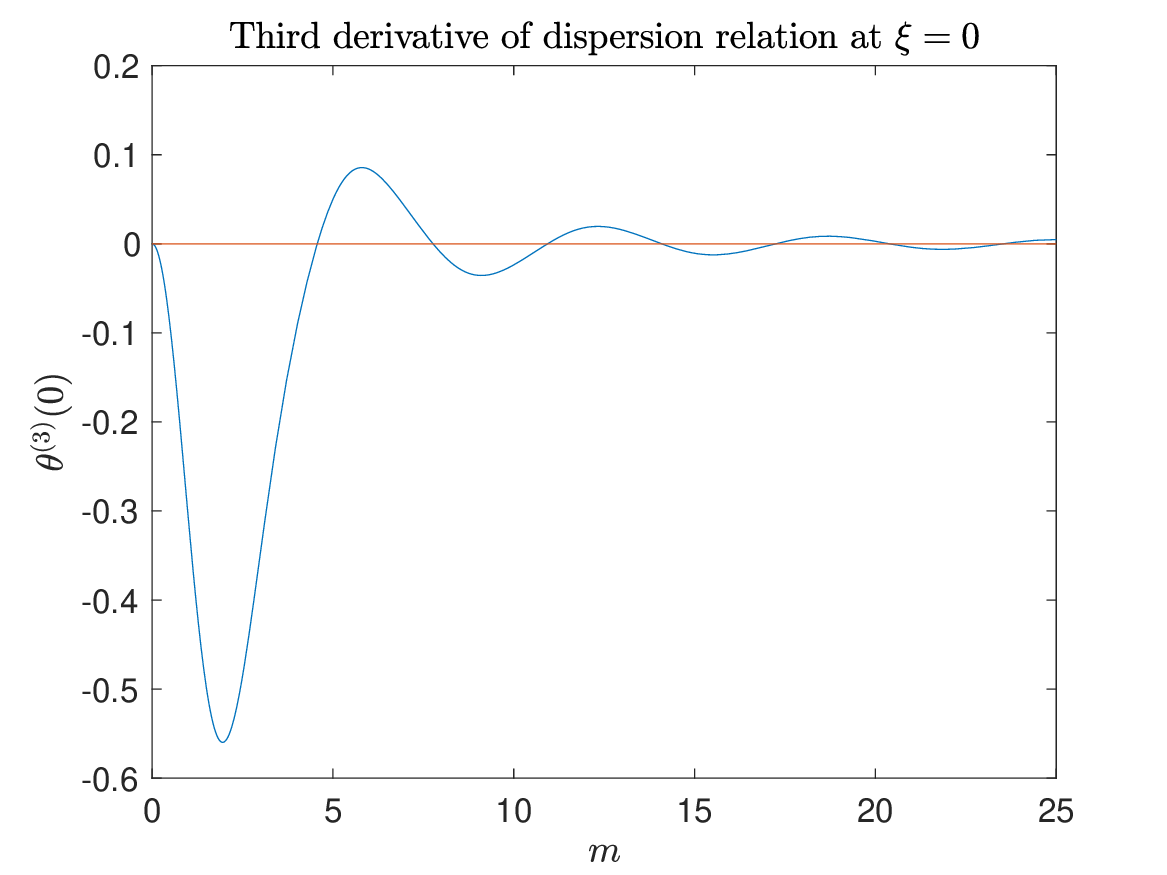}
        \caption{}
     \label{fig:theta3at0}     
    \end{subfigure}
    \hfill
    \begin{subfigure}{0.45\textwidth}
        \includegraphics[width=\linewidth]{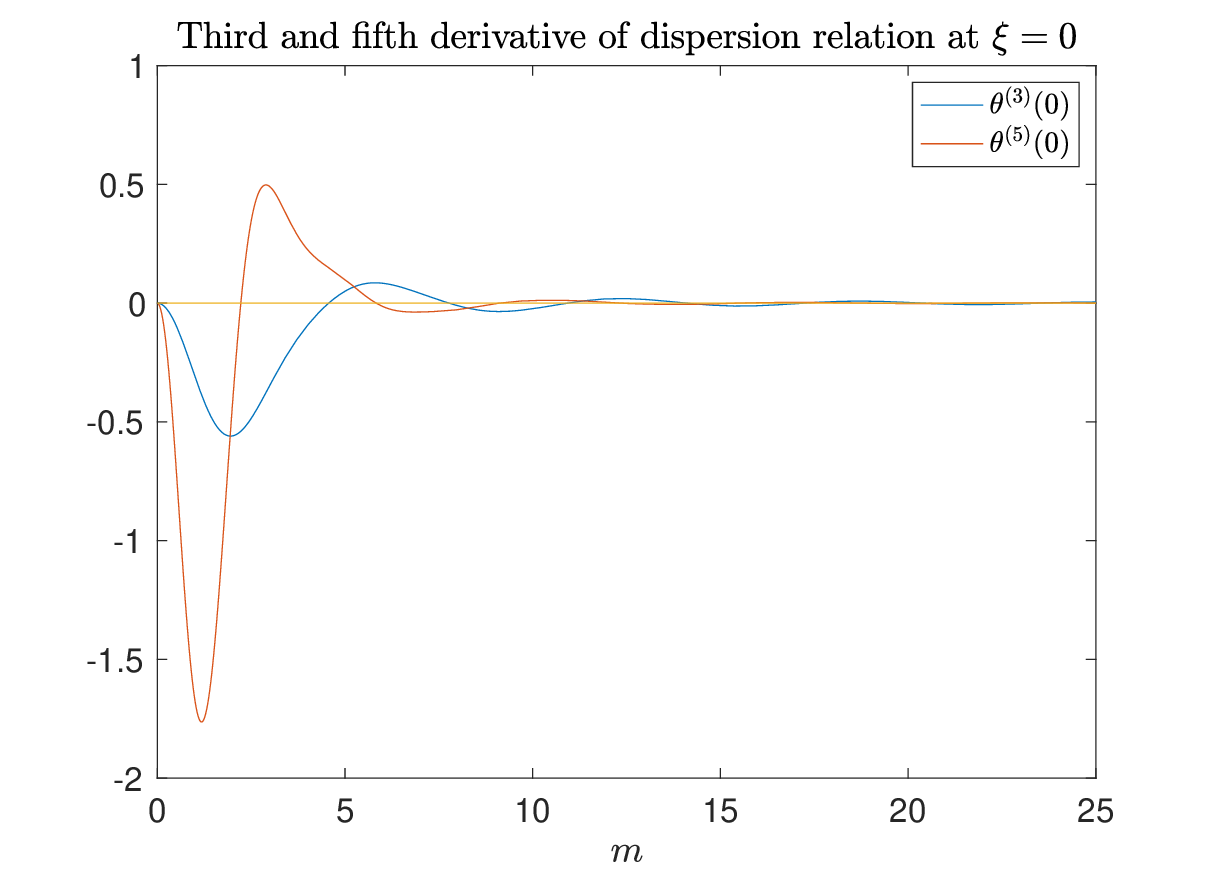}
        \caption{}
     \label{fig:theta3and5at0}     
    \end{subfigure}
    \caption{({\bf A}) $\theta ^{(3)}(0)$ as a function of $m>0$, where the dispersion relation $\theta (\xi)$ is given in \eqref{eq:theta}. ({\bf B}) $\theta ^{(3)}(0)$ and $\theta^{(5)}(0)$ overlaid.}  
 \end{figure}

Are there similar inflection points for larger values of $\xi$? Numerically, Fig.~\ref{fig:largeXi} demonstrates that when $m=1$ (a generic case, since $1\notin\Sigma$), there is a discrete sequence of Fourier-momenta $\{\xi _{l}\}$, tending to infinity, such that $\theta''(\xi_{l};1)=0$ and  $\theta '''(\xi_{l};1) \sim  \xi_l^{-2}\to0$ as $l\to\infty$. Then, assuming the observed decay of $\theta ''' (\xi_l,1)$, Van der Corput's Lemma implies   for data $f$ whose Fourier transform is localized near $\xi_l$, one has $\|M^n f\|_{\infty}\lesssim (\xi_l^{-2}n)^{-1/3}\|f\|_1$. Thus, for {\em general } $L^1 $ data, we conjecture:
\begin{figure}[h]
    \centering
    \begin{subfigure}{0.45\textwidth}
        \includegraphics[width=\linewidth]{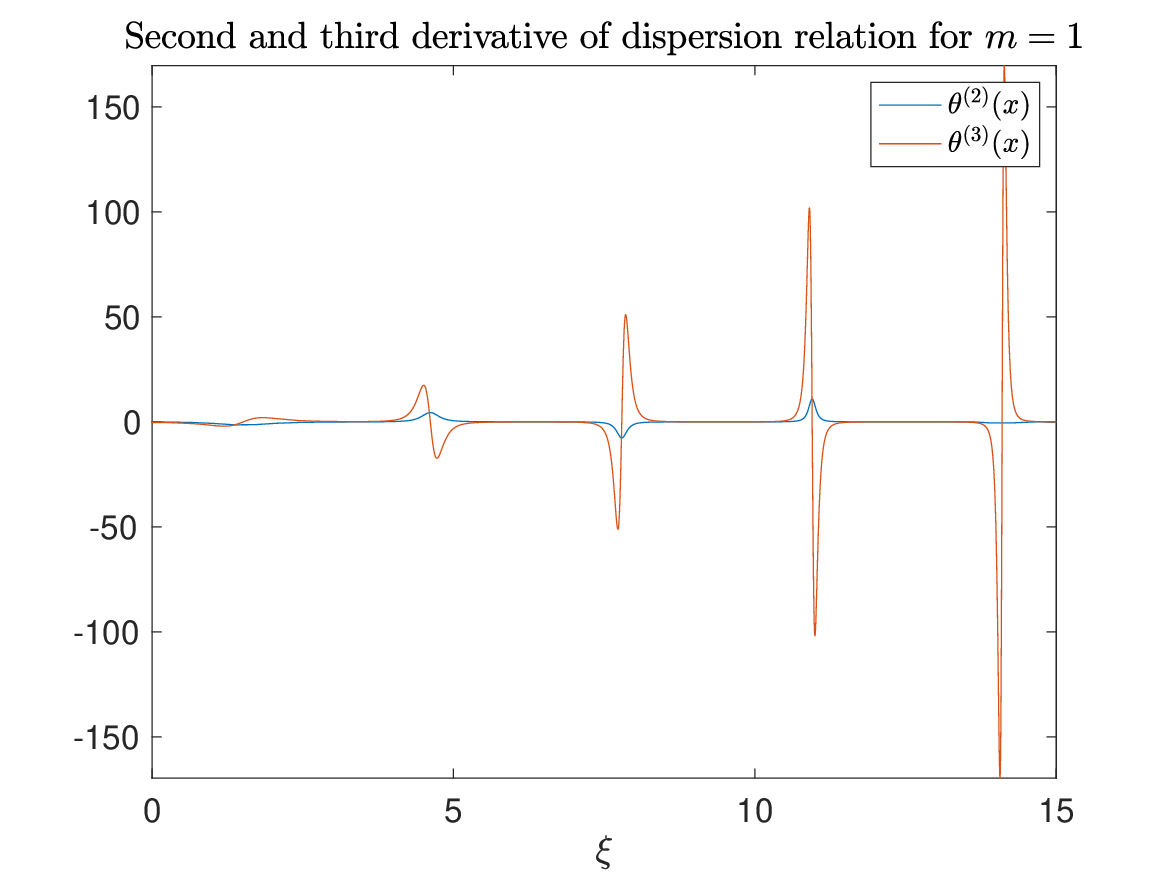}
        \caption{}
        \label{fig:subfig1}
    \end{subfigure}
    \hfill
    \begin{subfigure}{0.45\textwidth}
        \includegraphics[width=\linewidth]{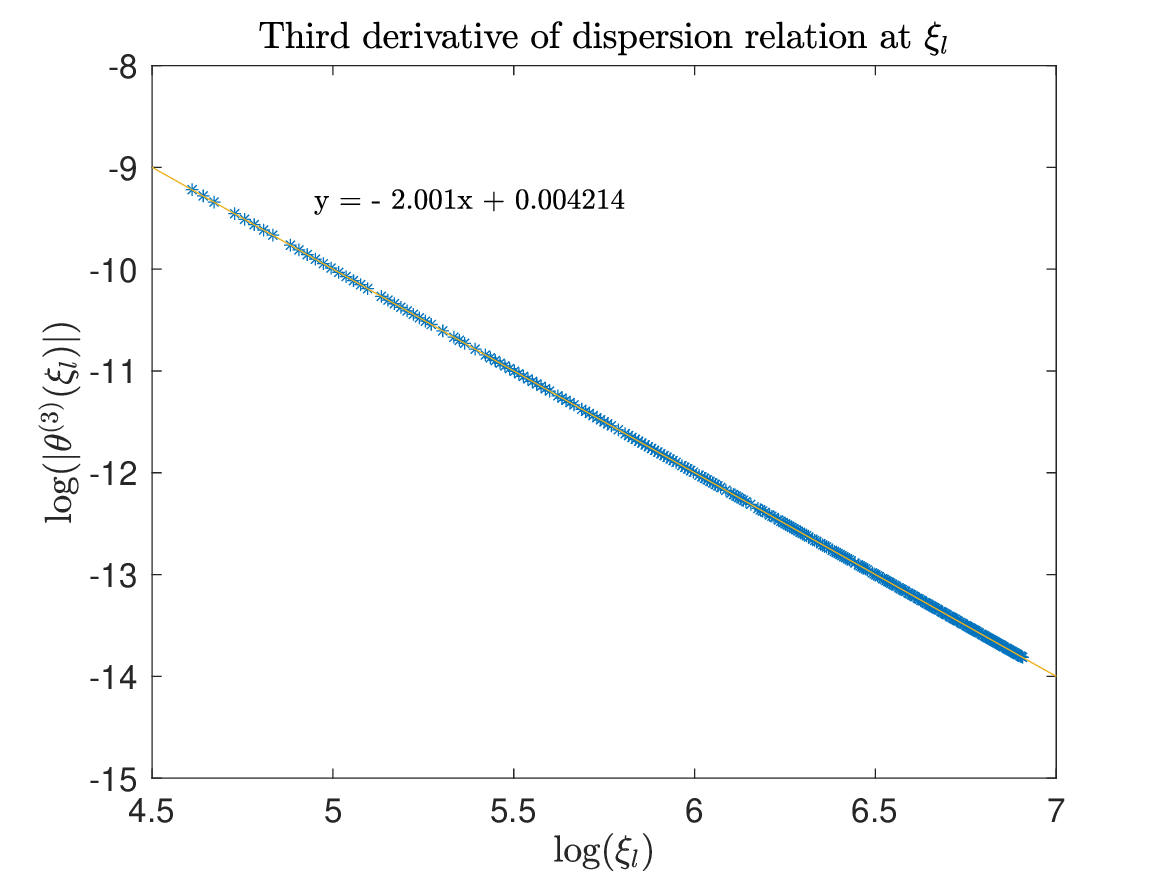}
        \caption{}
        \label{fig:subfig2}
    \end{subfigure}
    \caption{The dispersion relation $\theta (\xi)$, see \eqref{eq:theta}, for $m=1$. {\bf (A)} $\theta''(\xi)$ (blue) and $\theta '''(\xi)$ (orange). Each has an increasing sequence of zeroes. {\bf (B)} Denoting the zeroes of $\theta ''$ as $(\xi_l)_{l=1}^{\infty}$, we plot $\theta '''(\xi_l)$ (blue, stars) on a log-log grid, and a polynomial fit (orange, solid) which yields that $|\theta '''(\xi_l)|\lesssim \xi_l ^{-2}$.}
    \label{fig:largeXi}
\end{figure}

\begin{conjecture}\label{conjecture}
    Consider the Dirac equation \eqref{eq:pwc_mdirac} with $m=1$. Then for every $\varepsilon>0$ there exists $C_{\varepsilon}>0$ such that\footnote{The additional $\varepsilon>0$ smoothing arises in a dyadic partition argument; see for example, Section \ref{sec:thdecaypf}.}
    $$ \left\| M^n \langle\partial_x\rangle^{-(2/3+\varepsilon)} \right\|_{L^1\to L^{\infty}} \leq C_{\varepsilon}n ^{-\frac13}  \, .$$
\end{conjecture}

\subsection{The rotating-mass model}

Consider the Dirac equation
\begin{subequations}\label{eq:timeharmonic-H}
\begin{align} 
    i\partial_t \phi &= \cancel{D}_{\omega}(x)\phi =  (i\sigma_3\partial_x + \nu_{\omega}(t))\phi\, ,\quad t>0\, , x\in\R\, ,
\end{align}
where
\begin{align}
    \nu_{\omega}(t) \equiv m\begin{pmatrix}
        0 & e^{i\omega t}\\
        e^{-i\omega t} & 0
    \end{pmatrix} = m\left[ \cos(\omega t)\sigma_1 - \sin(\omega t)\sigma_2 \right] \, .
\end{align}
\end{subequations}

Let $u\mapsto\mathcal{U}_{\omega}(t)u$ denote the (unitary) time evolution operator associated to the dynamics \Cref{eq:timeharmonic-H}.
Our main technical result is that $\mathcal{U}_{\omega}(t)$ can be expressed in terms of the constant-mass Dirac time-evolution:
\begin{theorem}\label{Thm:timeharmonic-prop}
$\mathcal{U}_{\omega}(t)$ has the following Fourier integral representation:
    \begin{align}\label{eq:fourier-rep}
        \mathcal{U}_{\omega}(t)u(x) = \frac{1}{2\pi} \int_{\R}e^{i\xi x}e^{it\omega\sigma_3/2}e^{-i\cancel{D}_0(\xi+\omega/2)t}\hat{u}(\xi)d\xi\, .
    \end{align}
   Here,  $\cancel{D}_{0}(\xi)=(\xi\sigma_3 + m\sigma_1)$ is the symbol of the constant mass operator, and 
\begin{align*}
     e^{it\omega\sigma_3/2} = \begin{pmatrix}
        e^{+i\omega t/2} & 0 \\ 0 & e^{-i\omega t/2}
    \end{pmatrix}\, .
\end{align*}
\end{theorem}

\noindent
The above equivalence allows us to show that the time-decay of the constant mass equation dictates the same rate of decay to the time-harmonic \eqref{eq:timeharmonic-H}:
\begin{cor}\label{cor:thdecay}
For any $\eps > 0$
\begin{align}\label{eq:harmonic-decay}
    \|\mathcal{U}_{\omega}(t)\langle \partial_x\rangle^{-3/2-\eps}\|_{L^1\to L^{\infty}}\lesssim \langle t\rangle^{-1/2}\, .
\end{align}
\end{cor}

\section{Notation and Preliminaries}\label{sec:notation}
\begin{itemize}
\item The Fourier transform of a function $\alpha\in L^2(\R;\C^2)$ by
\begin{equation}\label{eq:FourierDef}
    \widehat\alpha(\xi) = \mathcal{F}[\alpha](\xi)= \int_{\R} e^{-i\xi x}\alpha(x)dx \, ,
\end{equation}
and its inverse is given by
 \begin{align*}
    \widecheck{\beta} (x)= \frac{1}{2\pi} \int_{\R} e^{+i\xi x}\beta(\xi)d\xi \, . 
\end{align*} 
\item The Laplace transform of a function $\alpha\in L^1((0,\infty))$ is defined
\begin{align*}
    \mathcal{L}[\alpha](s) = \int_0^{\infty} e^{-st}\alpha(t)dt,
\end{align*}
for $\re{s}>0$. 
\item The Pauli matrices are defined by $\sigma_0 = I$,
\begin{align}\label{eq:pauli}
    \sigma_1 = \begin{pmatrix}
        0 & 1 \\ 1 & 0
    \end{pmatrix} \, ,  \qquad \sigma_2 = \begin{pmatrix}
        0 & -i \\ i & 0
    \end{pmatrix} \, , \qquad \sigma_3 = \begin{pmatrix}
        1 & 0 \\ 0 & -1
    \end{pmatrix} \, .
\end{align}
\item Convention: Our Hamiltonians depend on a ``mass parameter'', $m$. We shall occasionally, when  convenient and when there is no ambiguity, suppress the $m$ dependence.
\end{itemize}

\section{The monodromy map \eqref{eq:monodromy} of the switching mass model}\label{app:monodromy-derivation}

We begin with a derivation of \Cref{eq:monodromy} for $\mathcal{U}(2)=\mathcal{U}_-(1)\mathcal{U}_+(1)$; see \cref{eq:Mdef-switch}.  Let 
\[ h(i\partial_x,m) \equiv (i\partial_x\sigma_3 + m\sigma_1) \, .\]
Starting with \cref{eq:pwc_mdirac}, we have via Fourier transform, that 
$\widehat\alpha(\xi,t)$ satisfies:
\begin{align*}
    i\partial_t\widehat\alpha &= h(-\xi,m) \widehat\alpha ,\quad 0\le t< 1 \, ,\\
    i\partial_t \widehat\alpha &= h(-\xi,-m)\widehat\alpha ,\quad 1\le t< 2 \, .
\end{align*}
Since different Pauli matrices anti-commute
\begin{align}
\sigma_3h(-\xi,m) = h(-\xi,-m)\sigma_3.
\label{eq:hcomm}\end{align}
The eigenpairs of $h(\xi,m)$ are:
\begin{align*}
  \lambda_{\pm}(\xi) = \pm \omega(\xi) \, ,\qquad   \bv_{\pm}(\xi) = 
    \rev{\frac{1}{\sqrt{n_{\pm}(\xi)}}} \begin{pmatrix}m \\ \xi\pm\omega(\xi) \end{pmatrix} \, ,
\end{align*}
where, as in \eqref{eq:theta}, $\omega (\xi;m)=\sqrt{\xi^2 +m^2}$, and $n_\pm(\xi)=2\omega(\xi)(\omega(\xi)\pm \xi)$ are normalization factors such that $\|\bv_{\pm}(\xi)\|=1$. Let $V(\xi)= V(\xi,m)= \rev{[\bv_+(\xi)\ \bv_-(\xi)]}$ denote the $2\times2$ matrix whose columns are
 $\bv_+(\xi)$ and $\bv_-(\xi)$. Since $h(-\xi,m)$ is Hermitian, $V(\xi,m)$ is unitary. Hence, 
 \[h(-\xi,m)V(\xi)=  V(\xi) \sigma_3\omega(\xi)\quad {\rm or}\quad 
 h(-\xi,m)=  V(\xi)\ \sigma_3\omega(\xi)\ V(\xi)^*\]
 Further, the commutation relation \cref{eq:hcomm} implies
 \[h(-\xi,-m)\ \sigma_3V(\xi)=  \sigma_3V(\xi) \sigma_3\omega(\xi)\quad {\rm or}\quad 
 h(-\xi,-m)\ =  \sigma_3V(\xi) \sigma_3\omega(\xi)\ V(\xi)^*\ \sigma_3.\]
The Fourier transform of the monodromy map,  $\widehat{M}(\xi)$,  is given by the product of unitary matrices:
\begin{align*}
\widehat{M}(\xi) &= e^{-ih(-\xi,-m)}\ e^{-ih(-\xi,m)} \\
&= \big(\ \sigma_3V(\xi) e^{-i\sigma_3\omega(\xi)}\ V(\xi)^*\ \sigma_3\big)\ 
\big( V(\xi) e^{-i\sigma_3\omega(\xi)}\ V(\xi)^*\ \big)\\
&= \big(\ \sigma_3V(\xi) e^{-i\sigma_3\omega(\xi)}\ V(\xi)^* \ \big)^2 
\  \equiv\ \mathscr{M}^2(\xi) \, .
\end{align*}

A direct calculation shows that 
\begin{align*}
    \mathscr{M}(\xi;m)&=\begin{pmatrix}
        \cos(\omega(\xi)) + i\xi \sinc(\omega(\xi))
          & -im\sinc(\omega(\xi)) \\
        im\sinc(\omega(\xi)) & - \cos(\omega(\xi)) +i\xi\sinc(\omega(\xi))  
    \end{pmatrix} \\
    &= i\xi \sinc(\omega(\xi))\sigma_0 + \cos(\omega(\xi))\sigma_3 +m\sinc(\omega(\xi))\sigma_2 \, ,
\end{align*}
\rev{where $\sinc(x) \equiv \sin(x)/x$.} Therefore, the eigenvalues of $\mathscr{M}(\xi;m)$ are
$$
\mu_\pm(\xi;m) = i\xi\sinc(\omega(\xi)) \pm\sqrt{\cos^2(\omega(\xi)) + m^2\sinc^2(\omega(\xi))} \, . $$
Note that $\mu_\pm(\xi)$ lie on the unit circle, as expected since $\widehat{M}(\xi)=\mathscr{M}^2(\xi)$ is unitary. Moreover, direct computation of $\mu_{\pm}^2(\xi)$, the eigenvalues of $\widehat{M}(\xi)$, shows that they are complex conjugate of one another.\footnote{This is to be expected by ODE theory \cite{coddington2012introduction}: since the right-hand side of the Fourier transformed \eqref{eq:tperDirac} has zero trace for all $t\in [0,T]$, the Floquet exponents have to sum up to $0$. Since $\widehat{M}(\xi;m)$ is unitary, the Floquet {\em multipliers} therefore are complex conjugates of each other.} Hence, we can write $\mu_{\pm}(\xi;m) = \exp \left[\pm i\theta(\xi;m) \right]$, where $\theta(\xi;m)$ is given, after some algebra, by \eqref{eq:theta}. The corresponding eigenvectors are given by
\begin{align} \label{eq:Pmatrix-elements}
    \bp_{\pm}(\xi;m) = \frac{1}{\sqrt{N(\xi;m)}}\begin{pmatrix}
    im\sinc(\omega(\xi;m)) \\ 
    \cos(\omega(\xi;m))\mp\sqrt{\cos^2(\omega(\xi;m))+m^2\sinc^2(\omega(\xi;m))}
\end{pmatrix}
\end{align}
where normalization factor $N(\xi)$ ensures $\|\bp_{\pm}\|=1$.
Defining the change of basis matrix $P(\xi;m)= \begin{pmatrix}
        \bp_{+}(\xi;m) & \bp_{-}(\xi;m)
    \end{pmatrix}$,
 it is clear that $P(\xi)$ is unitary and we have the Fourier representation of the monodromy given by
\begin{align*}
    \widehat{M}(\xi;m) = P(\xi;m)\begin{pmatrix}
     e^{+2 i \theta(\xi;m)}& 0 \\
     0& e^{-2 i \theta(\xi;m)}\end{pmatrix} P^* (\xi;m) \, . 
\end{align*}
Finally, inverting the Fourier transform, we obtain \eqref{eq:monodromy} as desired.

\section{Proof of Theorem \ref{thm:maintheorem}}\label{sec:mainPf}

Our proof of the dispersive time-decay bounds (Theorems \ref{thm:maintheorem} and \ref{thm:15theorem}) relies on the classical van der Corput  Lemma \cite{stein1993harmonic} :
    \begin{lemma}\label{lem:vandercorput}
    Let $\lambda$ be a smooth function and $f$ a smooth, compactly supported function. Suppose there exists $\lambda_0>0$, such that $\vert\lambda^{(k)}(z)\vert \geq \lambda_0$. Then there exists a constant, $c_k$, depending only on $k$, such that
    \begin{align}
        \left\vert\int_{\R} f(z)e^{i\lambda(z)} \, dz\right\vert \leq c_k \lambda_0^{-1/k}\| f'\|_{L^1}\ .
    \end{align}
\end{lemma}

By van der Corput Lemma, the decay properties of oscillatory integrals such as \eqref{eq:Mnf_integral} are intimately related to the points where the phase function, $\theta(\xi;m)$, and its derivatives vanish. By a direct calculation (see \Cref{app:thetaderivatives}),  $\theta''(0;m)=0$ for all $m\in(0,\infty)$. The following lemma shows that $\theta'''(0;m)\neq 0$ for all but a discrete set of values of $m$.

\begin{lemma}\label{lem:thetathirdder0}
    Let $\theta(\xi;m)$ be given by \eqref{eq:theta}. Then the vanishing set
    \begin{align*}
        \Sigma = \{ m\in(0,\infty)\ | \ \theta'''(0;m)=0\}
    \end{align*}
    is discrete. Writing $\Sigma = \{m_k\}_{k\geq 1}$, there exist $M>0$ and $k_0\in \mathbb{Z}$ such that 
     $m_k=(k+k_0+1/2)\pi + \mathcal{O}(k^{-1})$ for all $k\ge M$. \rev{Furthermore, the fifth derivative of $\theta$ does not vanish at these points, and in particular  \begin{equation}\label{eq:theta5_asym}
     m_k ^3 \cdot \theta^{(5)}(0;m_k)  = (-1)^{k+k_0+1}\cdot 15 +\mathcal{O}(k^{-2}) \,
     \end{equation} as $k\to \infty$.}
\end{lemma}
\begin{proof}
We have the explicit formula for the third derivative of $\theta$ evaluated at $\xi=0$ from \Cref{lem:thetader0}
    \begin{align}\label{eq:thetathirdder0}
        \theta ''' (0,m) &= \frac{1}{m^3}\left[ -2\sin^3(m)+ 3m\cos(m)-3\sin(m)\cos^2(m)\right]\, .
    \end{align}
    
   Hence,  $\theta ''' (0,m)$ vanishes if and only if $\cos(m)-(3m)^{-1}\sin(m) \cos^2(m)- 
   2(3m)^{-1}\sin(m)=0$. By analyticity, this equation has a discrete set of solutions. Furthermore, if we consider $m$ large, the solutions are precisely
    $m_k= (k+\frac12)\pi + \mathcal{O}(k^{-1})$, for $k\ge M$, where $M$ is sufficiently large.

    \rev{To prove \eqref{eq:theta5_asym}, first note that by simple Taylor expansion,  $\sin (m_k)= (-1)^{k+k_0}+ \mathcal{O}(k^{-2})$ and $\cos (m_k) = \mathcal{O}(k^{-1})$. Plugging these asymptotic expressions into \eqref{eq:theta5_exp}, the explicit formula for $\theta^{(5)}(0;m)$, leads to  the desired result.}
\end{proof}

\rev{While \eqref{eq:theta5_asym} is only valid for sufficiently large $m_k\in \Sigma$, Table \ref{tab:theta5} shows that for $k=1,\ldots ,8$, the agreement is quite good and $\theta^{(5)}$ does not vanish.

\begin{table}[h]
\begin{center}
\begin{tabular}{ |c|c|c|c|c|c|c|c|c| } 
 \hline
 $m_k$& 4.5659 & 7.7681 & 10.9346 & 14.0898 & 17.2401 & 20.3876 & 23.5336 & 26.6785\\
 \hline
 $m_k^3\theta^{(5)}(0;m_k) $& 14.1881 & -14.7151 & 14.8556 & -14.9129 & 14.9418 & -14.9583 & 14.9687 & -14.9757\\ 
 \hline
\end{tabular}
\end{center}
\caption{\rev{Numerically computed values of $m^3\theta^{(5)}(0;m)$ for $m=m_1,\ldots ,m_8 \ in\Sigma$, see \eqref{eq:SigmaDef}.}}
\label{tab:theta5}
\end{table}

}

We are now in a position to prove Theorems \ref{thm:maintheorem} and \ref{thm:15theorem}.

\subsection{Proof of the time-decay bound \eqref{eq:main_ubd} of Theorem \ref{thm:maintheorem}}\label{sec:mainOscil}
Fix $m\notin\Sigma$. By definition \eqref{eq:SigmaDef}, $\theta''(0;m)=0$ and $\theta'''(0,m)\neq 0$. Since $\theta'''(\xi;m)$ is continuous in $\xi$, there exist $c>0, \delta>0$ such that
\begin{align}\label{eq:DataProperties}
    \vert \theta'''(\xi;m)\vert > c=c(m) >0\, ,\qquad \text{for } \xi\in (-2\delta,2\delta)\, .
\end{align}
For the remainder of the proof we suppress the $m-$ dependence of $\theta$ and its derivatives. Let $\xi\mapsto\hat{f}(\xi)$ be smooth and supported in $[-\delta,\delta]$ and introduce  a smooth cutoff function $\xi\mapsto \chi(\xi)$,
 such that $\chi(\xi)\equiv1$ for $|\xi|\le\delta$ and $\chi(\xi)\equiv0$ for $|\xi|>1$.
Then,
\begin{align*}
    M^n f &=  \frac{1}{2\pi}\int_{\R} \left[\chi(\xi)P(\xi)\right] \begin{pmatrix}
     e^{+2 i n\theta(\xi)}& 0 \\
     0&  e^{-2 i n\theta(\xi)}\end{pmatrix} \left[\chi(\xi) P^* (\xi) \hat{f}(\xi)\right] e^{i\xi x} \, d\xi\\
     &=  \frac{1}{2\pi}\int_{\R} \chi(\xi)P(\xi) \begin{pmatrix}      e^{+2 i n\theta(\xi)}& 0 \\      0&  e^{-2 i n\theta(\xi)}\end{pmatrix}  \hat{\phi}(\xi)e^{i\xi x} \, d\xi\\
     &=\frac{1}{2\pi}\int_{\R} P^{\chi}(\xi)\begin{pmatrix}
        e^{+2 i n\theta(\xi)}\hat \phi_{+}(\xi) \\ e^{-2 i n\theta(\xi) }\hat \phi_{-}(\xi)
    \end{pmatrix}e^{i\xi x} \, d\xi \, ,
\end{align*}
where $P(\xi)$ is given in \eqref{eq:Pmatrix-elements}, 
\begin{equation}P^{\chi}(\xi)\equiv \chi(\xi)P(\xi)\quad {\rm and}\qquad  
  \hat\phi(\xi) = (\hat\phi_{+},\hat\phi_{-})^\top \equiv  (P^\chi)^*(\xi)\hat{f}(\xi)  \, .
  \label{eq:Pchi}
  \end{equation}

Hence, $M^nf$ is the sum of four terms, each of the form:
\begin{align}\label{eq:Mnf_postTri}
 I_{jk}(x;n) \equiv  \int_{\R} e^{i\xi x}\ P_{jk}^{\chi}(\xi) \left[e^{\pm 2 i n\theta(\xi)}\hat \phi_{\pm}(\xi)\right]\, d\xi = \left(\check{P}^{\chi}_{jk} * u_{\pm}\right)(x; n)\qquad  (j,k=1,2) \, , 
\end{align}
where
\begin{align}\label{eq:uDef}
    u_{\pm}(x,n) \equiv \int_{\R} e^{i\xi x}\ e^{\pm 2 i n\theta(\xi)}\hat \phi_{\pm}(\xi) \, d\xi ~~ .
\end{align}
It follows that 
\begin{align}
\sup_{x\in\R}\ |(M^nf)(x)| &\le \sum^2_{j,k=1} \sup_{x\in\R}\ |I_{jk}(x;n)| \ =\ \  \sum^2_{j,k=1}  \sup_{x\in\R}\ \Big|\left(\check{P}^{\chi}_{jk} * u_{\pm}\right)(x; n)\Big| \nonumber\\
&\le \rev{\left(\sum_{j,k=1}^{2}\|\check{P}^{\chi}_{jk}\|_{L^1}\right)}
\sup_{x\in\R} |u_{\pm}(x; n)|
\label{eq:sumIk-bound}\end{align}
We complete our bound on $M^nf$
 using the following estimate on the oscillatory integral \eqref{eq:uDef}:
\begin{lemma}\label{lem:decaywitoutP}
    \begin{align}\label{eq:u_pm-bound}
        \|u_{\pm}(\cdot, n)\|_{\infty} = \sup_{x} \left\vert\int_{\R} e^{\pm 2 i n\theta(\xi)+i\xi x}\hat \phi_{\pm}(\xi) d\xi \right\vert \lesssim \frac{1}{n^{1/3}}\|\phi_{\pm}\|_{L^{1}} \, .
    \end{align}
\end{lemma}
\begin{proof}
Using Fubini's Theorem and H{\"o}lder inequality, we have that
    \begin{align*}
        \sup_{x\in \mathbb{R}} \left\vert\int_{\R} e^{+ 2 i n\theta(\xi)+i\xi x}\hat \phi_{\pm}(\xi) \, d\xi \right\vert &= \sup_{x\in \mathbb{R}} \left\vert\int_{\R}\int_{\R} e^{+ 2 i n\theta(\xi)+i\xi (x-y)} \phi_{\pm}(y)  \, dy\, d\xi \right\vert \\
        &= \sup_{x\in \mathbb{R}} \left\vert\int_{\R}\left( \int_{\R} e^{+ 2 i n\theta(\xi)-i\xi (x-y)} \chi(\xi)   \,  d\xi \right)  \phi_{\pm}(y) \, dy \right\vert \\
        &\leq \sup_{s\in \R}\left\vert\int e^{in(\xi s +2 \theta(\xi)) }\chi(\xi)\, d\xi\right\vert\|\phi_{\pm}\|_{L^1} \, 
    \end{align*}
    where $\phi_{\pm} \in L^1(\R,\C)$ since $\phi_{\pm}$ are Schwartz class. Defining the phase function 
\begin{equation}\label{eq:BigTheta}
    \Theta (\xi, s) \equiv  \xi s + 2\theta(\xi) \, ,    
    \end{equation}
    where $\theta(\xi)$ is defined in \eqref{eq:theta}.
    Choose $s_0=-2\theta'(0)$ so that $\partial_{\xi}\Theta(0,s_0)=0$.
    Furthermore, since $\theta(\xi)$ is an odd function, we have $\partial_{\xi \xi}\Theta (0, s_0)=0$.   By van der Corput's Lemma, there is a constant $C>0$, independent of $\phi_{\pm}$, such that
    \begin{align*}
        \sup_{x} \left\vert\int_{\R} e^{\pm 2 i n\theta(\xi)+i\xi x}\hat \phi_{\pm}(\xi) \, d\xi \right\vert \leq C\frac{\|\partial_{\xi}\chi\|_{L^1}}{(cn)^{1/3}}\|\phi_{+}\|_{L^1} \, ,
    \end{align*}
    where $c= c(m)$ was chosen to satisfy the  bound $\vert \theta ''(\xi)\vert > c>0$ on the support of $\hat{\phi}_{\pm}$.
\end{proof} 

Substituting the bound of Lemma \ref{lem:decaywitoutP} into \eqref{eq:sumIk-bound}, we obtain
$\|M^nf\|_\infty \lesssim n^{-1/3} \|\phi\|_{L^1}   \| \check{P}^{\chi}_{jk}\|_{L^1}$.
Since $\| \check{P}^{\chi}_{jk}\|_{L^1}$ is finite and independent of $n$ and the data $f$, we have 
$    \|M^n f\|_{\infty} \lesssim n^{-1/3}
    \|\phi \|_{L^1}
    $.
Finally, $\hat\phi= (P^\chi)^* \hat{f} $  and so by Young's inequality
$  \| \phi \|_{L^1} = \left\| \check{P}^{\chi}\ast f\right\|_{L^1} \leq \|\check{P}^\chi\|_{L^1}\|f\|_{L^1} $. Therefore, $\|M^nf\|\lesssim n^{-1/3}\|f\|_{L^1}$. The proof of \eqref{eq:main_ubd} in \Cref{thm:maintheorem} is now  complete.

\subsection{Proof of the asymptotic expansion \eqref{eq:main_exp}}\label{sec:exp_pf}

We next further prove that for similar choice of initial data, $f$,  the time-decay upper bound in \eqref{eq:main_ubd} is attained. For this, we employ the asymptotic expansion \eqref{eq:exp_airy} given in \Cref{lem:Airyexpansion}.

As above in the proof of \Cref{lem:decaywitoutP}, we choose  $s_0 = -2\theta'(0)$. Once again,  $\partial_{\xi}\Theta(0,s_0) = 0$ and 
    $\partial_{\xi \xi }\Theta(0,s_0) = 0$, and furthermore  we find via  \eqref{eq:BigTheta} and \eqref{eq:theta}, that $\partial_{\xi \xi \xi}\Theta(0,s_0) = \theta'''(0)\neq 0$.
Thus, for every time $t=nT$, there is a point $x=ns_0$ in which we can apply the asymptotic expansion \Cref{lem:Airyexpansion} to \eqref{eq:Mnf_integral}, the integral representation of $M^n f (x)$, yielding
\begin{align*}
 [M^n f] (ns_0) &= \frac{1}{2\pi}    \int_{\R} P(\xi) \begin{pmatrix}
     e^{+2 i n\theta(\xi)}& 0 \\
     0&  e^{-2 i n\theta(\xi)}\end{pmatrix} P^* (\xi) \hat{f}(\xi)e^{i\xi n s_0} \, d\xi  \\
    [{\rm \Cref{lem:Airyexpansion}}] \qquad \qquad   &= \frac{1}{2\pi}P^{\chi}(0) \begin{pmatrix}
          \hat{\phi}_{+}(0) \\ \hat{\phi}_{-}(0)
     \end{pmatrix}\Ai(0)\vert \theta'''(0)\vert^{1/3}e^{\pm 2i\theta(0)n}\frac{1}{(3n)^{1/3}} + O(n^{-2/3}) \\
    [{\rm \cref{eq:Pchi}}] \qquad \qquad &=\frac{1}{2\pi} P^\chi(0) (P^\chi)^*(0)\begin{pmatrix}
    \hat{f}_1(0) \\ \hat{f}_2 (0)
\end{pmatrix} \Ai(0)\vert \theta'''(0)\vert^{1/3}e^{\pm 2i\theta(0)n}\frac{1}{(3n)^{1/3}} + O(n^{-2/3}) \\
    &= \frac{1}{2\pi}\begin{pmatrix}
    \int\limits_{\R} f_1(x) \, dx \\ \int\limits_{\R} f_2 (x) \, dx
\end{pmatrix} \Ai(0)\vert \theta'''(0)\vert^{1/3}e^{\pm 2i\theta(0)n}\frac{1}{(3n)^{1/3}} + O(n^{-2/3}) \, .
\end{align*}

\subsection{Proof of Theorem \ref{thm:15theorem}}\label{sec:15pf}

The proof  for the case of exceptional mass parameters, $m\in\Sigma=\{m_k\}_{k\ge1}$, is  analogous  to that of Theorem \ref{thm:maintheorem}. Recall that by definition,  $\theta'''(0;m_k)=0$. Also, since $\theta(\xi)$ is an odd smooth function,  $\theta ^{(4)}(0;m)=0$ for {\em all} $m>0$. \rev{Furthermore, by \eqref{eq:theta5_asym} we know that $\theta^{(5)}(0,m_k)\neq 0$ for all but (perhaps) finitely many $m_k \in \Sigma$, and numerical evidence in \Cref{fig:fifth_derivative_at_masses} implies that in fact $\theta^{(5)}(0,m_k)\neq 0$ for  {\em all} values of $m_k$. }

\begin{figure}
    \centering
    \includegraphics[width=0.5\linewidth]{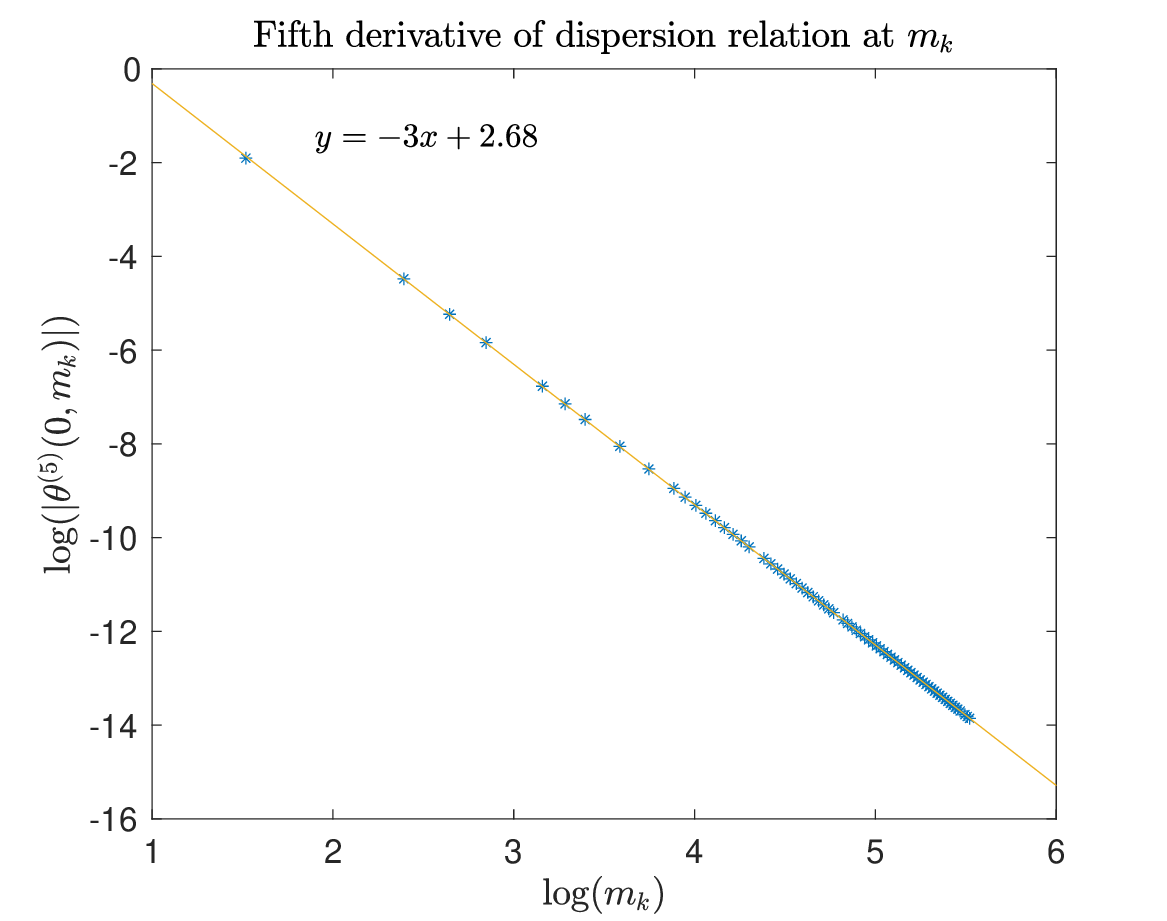}
    \caption{Log-log plot of $\theta^{(5)}(0,m_k)$ which shows that $\theta^{(5)}(0,m_k)\sim m_k^{-3}$.}
    \label{fig:fifth_derivative_at_masses}
\end{figure}

Hence, the upper bound part of the proof follows in the exact same way as in Theorem \ref{thm:maintheorem}, only with an upper bound of the form  $$\|u_{\pm}(\cdot, n)\|_{\infty} = \sup_{x} \left\vert\int_{\R} e^{\pm 2 i n\theta(\xi)+i\xi x}\hat \phi_{\pm}(\xi) d\xi \right\vert \lesssim \frac{1}{n^{1/5}}\|\phi_{\pm}\|_{L^{1}} \, ,$$ instead of the analogous result in Lemma \ref{lem:decaywitoutP}, which is proven by Van der Corput Lemma (Lemma \ref{lem:vandercorput}) in an analogous way.

The expansion argument is identical to that of Sec.\ \ref{sec:exp_pf}. The only difference that, because the phase $\Theta (\xi,s_0)$ is now {\em  triply} degenerate at $\xi=0$, we use the asymptotic expansion \eqref{eq:exp_quint} given in Lemma \ref{lem:Airyexpansion} instead, and the proof follows.

\section{Analysis of the rotating mass model \eqref{eq:timeharmonic-H}; proof  of  \Cref{Thm:timeharmonic-prop} }\label{sec:thpf}

To obtain the formulas in \Cref{Thm:timeharmonic-prop} we apply both the Fourier (in $x$) and Laplace (in $t$) transforms to the solution $(\alpha_1 (t,x),\alpha_2(t,x))^{\top}$ of \eqref{eq:timeharmonic-H}, and solve the corresponding algebraic system exactly. Let
$$
    \Phi_j(\xi,s) \equiv \mathcal{L}[\mathcal{F}[\alpha_j]](\xi,s)\, , \qquad j=1,2 \, .$$
The the transformed \eqref{eq:timeharmonic-H} is
\begin{align}
    is\Phi_1(\xi,s) -i\hat u_1 &= +\xi \Phi_1(\xi,s) + m \Phi_2(\xi,s-i\omega) \label{eq:FL1}\, ,\\
    is\Phi_2(\xi,s) -i\hat u_2 &=-\xi \Phi_2(\xi,s) + m \Phi_1(\xi,s+i\omega)\label{eq:FL2}\, ,
\end{align}
where $(\hat{u}_1, \hat{u}_2)$ is the initial data. By the replacement $s\mapsto s+i\omega$ in the \eqref{eq:FL1} we obtain $\Phi_1(\xi,s+i\omega)$ in terms of $\Phi_2(\xi,s)$. Subsequently substituting this expression into the \eqref{eq:FL2} gives a single equation for $\Phi_2(\xi,s)$ which is easily solved. Finally an expression for $\Phi_1(\xi,s)$ is then obtained from \eqref{eq:FL1} and the expression for $\Phi_2(\xi,s)$.
\begin{align*}
    \Phi_1(\xi,s) &= \frac{(s-i\omega/2)}{(s-i\omega/2)^2 + (p(\xi))^2}\hat u_1(\xi) - \rev{\frac{p(\xi)}{(s-i\omega/2)^2 + (p(\xi))^2}}\left(\frac{im\hat u_2(\xi)+i(\xi+\omega/2)\hat u_1(\xi)}{p(\xi)}\right)\, , \\
     \Phi_2(\xi,s) &= \frac{(s+i\omega/2)}{(s+i\omega/2)^2+(p(\xi))^2}\hat u_2(\xi) + \frac{p(\xi)}{(s+i\omega/2)^2+(p(\xi))^2}\left(\frac{i(\xi+\omega/2)\hat u_2(\xi)-im\hat u_1(\xi)}{p(\xi)}\right)\, ,
\end{align*}
where $
    p(\xi) = \sqrt{(\xi+\omega/2)^2+m^2}>0$. 
Inverting the Laplace transform first and then the Fourier transform we obtain
\begin{align*}
    \phi_1(x,t) &= e^{+i\omega t/2}\int_{\R} e^{i\xi x}\left(\cos(p(\xi)t)\hat u_1(\xi) -\frac{\sin(p(\xi)t)}{p(\xi)}\left[im\hat u_2(\xi)+i(\xi+\omega/2)\hat u_1(\xi)\right]\right) \frac{d\xi}{2\pi}\, ,\\
     \phi_2(x,t) &= e^{-i\omega t/2}\int_{\R}e^{i\xi x}\left(\cos(p(\xi)t)\hat u_2(\xi) + \frac{\sin(p(\xi)t)}{p(\xi)}\left[i(\xi+\omega/2)\hat u_2(\xi)-im\hat u_1(\xi)\right]\right)\frac{d\xi}{2\pi}\, .
\end{align*}
This can be written succinctly in the matrix form
\begin{subequations}\label{eq:prop}
\begin{align}
    \phi(x,t) = \frac{1}{2\pi}\int_{\R} e^{i\xi x} U(\xi,t)\hat{\phi}_0(\xi)d\xi\ ,
\end{align}
with the matrix $U(\xi,t)$ defined
\begin{align}
U(\xi,t)=
\begin{pmatrix}
    e^{i\omega t/2}\left(\cos(p(\xi)t)-i(\xi+\omega/2)\frac{\sin(p(\xi)t)}{p(\xi)}\right) & -im e^{i\omega t/2}\frac{\sin(p(\xi)t)}{p(\xi)}\\
    -im e^{-i\omega t/2}\frac{\sin(p(\xi)t)}{p(\xi)} & e^{-i\omega t/2}\left(\cos(p(\xi)t) + i(\xi+\omega/2)\frac{\sin(p(\xi)t)}{p(\xi)}\right)
\end{pmatrix}
\end{align}
\end{subequations}
We note that $    U(\xi,t) =  e^{it\omega\sigma_3/2}e^{-i\cancel{D}_0(\xi+\omega/2)t}$. 
where $e^{-i\cancel{D}_0(\xi)t}$ is the Fourier propagator of the constant mass Dirac equation. This relation combined with \Cref{eq:prop} proves the theorem.

\subsection{Proof of Corollary \ref{cor:thdecay}}\label{sec:thdecaypf}

The following argument appears in \cite{Erdogan21, kraisler2023dispersive}, and is included here briefly, for completeness. Fix $\eps > 0$. This estimate follows by considering \Cref{eq:fourier-rep}
\begin{align*}
    \|\mathcal{U}_{\omega}(t)\langle \partial_x\rangle^{-3/2-\eps}u(\cdot)\|_{L^{\infty}} & = \sup\limits_{x}\left\vert \int_{\R}e^{i\xi x}e^{it\omega\sigma_3/2}e^{-i\cancel{D}_0(\xi+\omega/2)t}\langle \xi\rangle^{-3/2-\eps}\hat{u}(\xi)\frac{d\xi}{2\pi}\right\vert\\
    &=\sup\limits_{x}\left\vert \int_{\R}\mathcal{K}(x-y,t)u(y)dy\right\vert\ \leq \|\mathcal{K}(\cdot,t)\|_{L^\infty} \cdot \|u\|_{L^1}\, ,
\end{align*}
where the kernel $\mathcal{K}(r,t)$ is given by
\begin{align}\label{eq:kernelToEstimate}
    \mathcal{K}(r,t)=\int_{\R}e^{i\xi r}e^{it\omega\sigma_3/2}e^{-i\cancel{D}_0(\xi+\omega/2)t}\langle \xi\rangle^{-3/2-\eps}\frac{d\xi}{2\pi}\, .
\end{align}
Thus the proof of the estimate \Cref{eq:harmonic-decay} reduces to showing that the kernel function $\mathcal{K}(r,t)$ has the desired decay. This follows from a van der Corput Lemma-type argument, applied to dyadic cutoff functions, similar to ~\cite[Theorem 2.3]{Erdogan21}. We sketch the proof here. 

For $j\in\N$, let $\psi_j\in C_c^{\infty}(\R)$ with $\supp{\psi_j}\subset [2^{j-1}-\omega/2,2^{j+1}-\omega/2]$ and let $\psi_0\in C_c^{\infty}(\R)$ be supported in a small neighborhood around $-\omega/2$ such that
\begin{align*}
     \sum_{j=0}^{\infty} \psi_j = 1\, , \qquad \|\psi_j\|_{L^1}\lesssim 2^j\, ,\qquad \|\partial_{\xi}\psi_j\|_{L^1}\lesssim 1\, .
\end{align*}
By inserting this partition of unity under the integral sign in \eqref{eq:kernelToEstimate} we see
\begin{align}\label{eq:kerneldecomp}
    \vert \mathcal{K}(r,t)\vert \leq \sum_{j=0}^{\infty} 2^{-3j/2} 2^{-\eps j} I_j\, ,
\end{align}
where
\begin{align*}
    I_j= \left\vert\int_{\R} e^{\pm i(\sqrt{(\xi+\omega/2)^2+m^2}+\omega/2)t-irx}\psi_j(\xi) d\xi\right\vert\, .
\end{align*}
By an application of the Van der Corput lemma, \ref{lem:vandercorput}, along with the inequality
\begin{align*}
    \left| \partial _{kk} \left[t(\sqrt{(\xi+\omega/2)^2+m^2}+\omega/2)+kr\right] \right| &=  \frac{tm^2}{((\xi+\omega/2)^2+m^2)^{\frac32}}\gtrsim t2^{-3(j+2)} \, ,
\end{align*}
which holds on $\supp \psi_j$, we observe
\begin{align*}
    I_j&\leq C\min\left(\|\psi\|_{L^1},\vert t\vert^{-\frac12}2^{\frac{3}{2}j}\ \|\partial_{\xi}\psi\|_{L^1}\right) \, .
\end{align*}
\Cref{cor:thdecay} then follows from combining the bounds above with the decomposition \eqref{eq:kerneldecomp}.

\appendix

\section{Statement and Proofs of Asymptotic expansions}\label{app:asymptotic- expansion}

In this appendix we provide the statement and proof of two asymptotic expansions used in the proofs of Theorem \ref{thm:maintheorem} and Theorem \ref{thm:15theorem}. Expansion \eqref{eq:exp_airy} is adapted from \cite[Equation 7.7.29]{hormander2007analysis}.

\begin{lemma}\label{lem:Airyexpansion}
  Let $\lambda$ be a smooth function and $f$ a smooth and compactly supported.       
\begin{enumerate}
    \item

     If $\lambda'(0)=\lambda''(0)=0$, but $\lambda'''(0)\neq0$, then as $\omega \to \infty$,
    \begin{align}\label{eq:exp_airy}
        \int_{\R} f(z)e^{i\omega\lambda(z)} \, dz =  2\pi e^{i\lambda(0)\omega}\Ai(0)f(0)\left(\frac{2}{\vert\lambda'''(0)\vert}\right)^{\frac13}\omega ^{-\frac13} + O(\omega^{-2/3}) \, , 
    \end{align}
    where $\Ai (x)$ is the Airy function (of the first kind). 
    \item 
    If $\lambda^{(j)}(0)=0$ for $j=1,2,3,4$, but $\lambda ^{(5)}(0)\neq 0$, then as $\omega \to \infty$,
     \begin{equation}\label{eq:exp_quint}
         \int_{\R} f(z)e^{i\omega\lambda(z)} \, dz =e^{i\lambda(0)\omega}\frac25 \Gamma \left(\frac15\right)\sin \left( \frac{2\pi}{5}\right) \left(\frac{120}{|\lambda ^{(5)}(0)|}\right)^{-\frac15} f(0) \omega ^{-\frac15} + O(\omega ^{-\frac25}) \, ,
     \end{equation}
     where $\Gamma (z)=\int_0^{\infty} t^{z-1}e^{-t} \, dt$ is the usual Gamma Function.
    \end{enumerate}
\end{lemma}

\begin{proof}
Recall that the Airy function $\Ai(x)$ is defined as the solution to the boundary value problem
\begin{align*}
    y''(x)-xy =0 \, , \qquad 
    y(0) = \frac{1}{3^{2/3}\Gamma(2/3)}\, , \qquad 
\lim_{x\to\infty}y(x) = 0 \, .
\end{align*}
By Fourier transforming the Airy equation with respect to $x$, we can verify that
\begin{align*}
    \widehat\Ai(\xi) = e^{-i\xi^3/3} \, .
\end{align*}
Now suppose $f\in C_c^{\infty}(\R)$ and $\lambda\in C^{\infty}(\R)$ such that $\lambda'(0)=\lambda''(0)=0$, but $\lambda'''(0)\neq 0$. Then there exists $a\in C^{\infty}(\R)$ such that $a(0)\neq 0$ and
\begin{align}\label{eq:az_def}
    \lambda(z) = \lambda(0) + z^3a(z)\, .
\end{align}
Introducing the change of variables 
\begin{align*}
    \zeta = z\vert \alpha(z)\vert^{1/3}\, , 
\end{align*}
the integral of interest becomes

\begin{align}\label{eq:betazeta}
        \int_{\R} f(z)e^{i\omega\lambda(z)} \,  dz = e^{i\lambda(0)\omega }\int_{\R} f(\beta(\zeta))e^{i\omega \zeta^3}\frac{dz}{d\zeta} (\beta (\zeta)) \, d\zeta\, .
\end{align}
where $\beta(\zeta)$ is the inverse change of variables, i.e., $\beta (\zeta (z))=z$. By direct substitution into the definition of $\zeta (z)$, the only solution to the equation $\zeta =0$ is $z=0$, thus $\beta(0)=0$. Since
$$\frac{d\zeta}{dz}=\frac13 z |a(z)|^{-2/3} + |a(z)|^{1/3} \, ,$$
To compute the limit as $z\to 0$, we use \eqref{eq:az_def} and the inverse Function Theorem to get
\begin{align}
    \frac{d\zeta}{dz}(0) = \left[\frac{dz}{d\zeta}(0)\right]^{-1} = \vert a(0)\vert^{-1/3} = \left(\frac{\vert\lambda'''(0) \vert}{6}\right)^{-1/3} >0\, ,
\end{align}
and so the inverse change of variables $\beta (\zeta)$ is well-defined.

Going back to \eqref{eq:betazeta}, we apply the Plancherel's theorem
\begin{align*}
    \int_{\R} f(x)\overline{g(x)}\, dx = \frac{1}{2\pi}\int_{\R} \hat f(\xi)\overline{\hat g(\xi)}\, d\xi\, ,
\end{align*}
and the first order Taylor expansion of the Airy function
\begin{align*}
    \Ai(\xi) = \Ai(0) + O(\xi)\, ,
\end{align*}
to obtain
\begin{align}
    e^{i\lambda(0)\omega}\int_{\R} &f(\beta(\zeta))e^{i\omega \zeta^3}\frac{d\zeta}{dz}(\beta (\zeta)) \, d\zeta \\
    &= {2\pi} e^{i\lambda(0)\omega }\int_{\R} \widecheck{\left(f\cdot \frac{d\zeta}{dz}\right) \circ \beta }(y) \cdot \Ai\left(\frac{y}{(3\omega)^{1/3}}\right)(3\omega)^{-1/3} \, dy \\
     &= {2\pi} e^{i\lambda(0)\omega }\int_{\R} \widecheck{\left(f\cdot \frac{d\zeta}{dz}\right) \circ \beta }(y) \cdot \left(\Ai\left(0\right) + O(\omega^{-1/3})\right)(3\omega)^{-1/3} \, dy \\
      &= {2\pi} e^{i\lambda(0)\omega }\frac{\Ai(0)}{(3\omega)^{1/3}}\int_{\R} \widecheck{\left(f\cdot \frac{d\zeta}{dz}\right) \circ \beta }(y) \, dy +  O(\omega^{-2/3}) \\
    &= {2\pi} e^{i\lambda(0)\omega }A(0)\frac{f(0)\cdot \frac{d\zeta}{dz}(0)}{(3\omega)^{1/3}}+O(\omega^{-2/3})\\
    &=2\pi e^{i\lambda(0)\omega}\Ai(0)f(0)\left(\frac{2}{\vert\lambda'''(0)\vert}\right)^{\frac13}\omega ^{-\frac13} + O(\omega^{-2/3})\, .
\end{align}
This completes the proof of the doubly degenerate stationary phase \eqref{eq:exp_airy}.

Now, for the triply degenerate case \eqref{eq:exp_quint}, where $\lambda ^{(j)}(0)=0$ for $j=1,\ldots, 4$, but $\lambda ^{(5)}(0)\neq 0$, write $\lambda(z) = \lambda (0) + a(z)z^5$, and define the analogous change of variables  $\eta \equiv z|a(z)|^{1/5}$. Setting now $\beta (\eta)$ as the inverse change of variables, i.e., $\beta (\eta (z))=z$, we can write 
\begin{align*}
    \int\limits_{\R} f(z)e^{i\omega \lambda (z)} \, dz &= e^{i\omega\lambda(0)}\int\limits_{\R} \left(f\circ \beta (\eta) \cdot \frac{dz}{d\eta}\right) e^{i\omega \eta ^5} \, d\eta  \, .
\end{align*}
Denoting $u(\eta) \equiv f\circ \beta (\eta) \cdot \frac{dz}{d\eta} $, here we can use the expansion from \cite[Equation (7.7.30)]{hormander2007analysis}, which to zeroth order reads
\begin{align*}
    \int\limits_{\R} u(\eta) e^{i\omega \eta ^4} \, d\eta &= \frac25 \Gamma \left(\frac15\right)\sin \left( \frac{2\pi}{5}\right)u(0)\omega ^{-\frac15} + O(\omega ^{-\frac25}) \\
    &= \frac25 \Gamma \left(\frac15\right)\sin \left( \frac{2\pi}{5}\right)f(\beta (0)) \frac{dz}{d\eta} (\beta (0)) \omega ^{-\frac15} + O(\omega ^{-\frac25}) \, .
\end{align*}
Here again, $\beta (0)=0$ since the only solution to the equation $\eta(z)=0$ is $z=0$, here again using the Inverse Function Theorem, we get
\begin{align*}
    \int\limits_{\R} u(\eta) e^{i\omega \eta ^4} \, d\eta = \cdots &= \frac25 \Gamma \left(\frac15\right)\sin \left( \frac{2\pi}{5}\right) \left(\frac{120}{|\lambda ^{(5)}(0)|}\right)^{-\frac15} f(0) \omega ^{-\frac15} + O(\omega ^{-\frac25}) \, .
\end{align*}
\end{proof}

\section{Derivatives of $\theta(\xi,m)$ at $\xi=0$.}\label{app:thetaderivatives}

This appendix tabulates the function $\theta(\xi,m)$ for the switching-mass model \eqref{eq:pwc_mdirac}, and its derivatives at $\xi=0$. All expressions can be obtained by direct differentiation of \eqref{eq:theta}.
\begin{lemma}\label{lem:thetader0}
    For all $m > 0$ 
    \begin{align}
        \theta'(0,m) &= \frac{\sin\left(m\right)}{m}\, ,
        \\
        \theta '' (0,m) &=0 \, ,\\
        \theta ''' (0,m) &= \frac{1}{m^3}\left[ -2\sin^3(m)+ 3m\cos(m)-3\sin(m)\cos^2(m)\right]\, ,\\
        \theta^{(4)}(0,m) &=0\, , \\
        \theta^{(5)}(0,m) &=\frac{3}{m^5}\left[-5m\cos(m)+3\sin(m)\cos^4(m)+12\sin(m)-10m\cos^3(m)-5m^2\sin(m)-4\sin^3(m)\right] \, . \label{eq:theta5_exp}
    \end{align}
\end{lemma}

\nocite{*}
\printbibliography

\end{document}